\documentclass{amsart}

\usepackage{amsfonts,amsthm,latexsym,amsmath,amssymb,amscd,amsmath, mathrsfs, epsf}
\usepackage{graphicx}
\newtheorem{theorem}{Theorem}
\newtheorem{proposition}[theorem]{Proposition}
\newtheorem{lemma}[theorem]{Lemma}
\newtheorem{definition}{Definition}
\newtheorem{corollary}{Corollary}
\newtheorem{example}{Example}
\newtheorem{remark}{Remark}

\newtheorem{conjecture}{Conjecture}

\newcommand{\bR}{\mathbb{R}}
\newcommand{\bC}{\mathbb{C}}
\newcommand{\ee}{\end{equation}}
\newcommand {\al}{\alpha}
\newcommand {\la}{\lambda}
\newcommand{\be}{\beta}
\newcommand {\ga}{\gamma}
\newcommand {\Ga}{\Gamma}

\newcommand {\bCP} {\mathbb {CP}}

\newcommand \dq {\mathfrak d}

\newcommand{\D}{\mathcal D}
\newcommand{\C}{\mathcal C}

\newcommand{\N}{\mathcal N}

\newcommand \M {\mathcal {M}}

\newcommand \prt {\partial}

\begin{document}
          \numberwithin{equation}{section}

          \title[On motherbody measures with algebraic Cauchy transform]
          {On motherbody measures with algebraic Cauchy transform}

\author[R.~B\o gvad]{Rikard B\o gvad}
\address{Department of Mathematics, Stockholm University, SE-106 91
Stockholm,
         Sweden}
\email{rikard@math.su.se, shapiro@math.su.se}

\author[B.~Shapiro]{Boris Shapiro}

\dedicatory{To Harold S. Shapiro}
\date{\today}

\begin{abstract} Below we discuss the existence of a motherbody measure for the  exterior inverse problem in potential theory in the complex plane. More exactly, we study  the question of representability almost everywhere (a.e.) in $\bC$ of (a  branch of) an irreducible  algebraic function  as the Cauchy transform of a signed measure  supported on a finite number of compact semi-analytic curves and a finite number of isolated points.  
  Firstly,  we present a large class of  algebraic functions for which there (conjecturally) always exists a positive measure with the above properties. This class was  discovered in our earlier study 
   of exactly solvable linear differential operators. Secondly,  we  investigate in detail the representability problem  in the case when the Cauchy transform satisfies  a quadratic equation with polynomial coefficients a.e. in $\bC$.  Several conjectures and open problems are posed. 
\end{abstract}

\maketitle

\section{Introduction}  

The study of local and global properties of the Cauchy transform and the Cauchy-Stieltjes integral was initiated by A.~Cauchy and T.~Stieltjes in the middle of the 19th century. 
   Large numbers  of papers and several books are partially or completely devoted to this area 
   which  is  closely connected with the  potential theory in  the complex plane and, especially, to the  inverse  problem and to the inverse moment problem, see e.g. \cite{Be}, \cite{CMR}, 
    \cite{Ga}, 
     \cite{Mu}, \cite{Za}.

During the last  decades the notion of a motherbody of a solid domain or, more generally, of a positive Borel measure  was  discussed both in  geophysics and mathematics, see e.g. 
\cite{Sj}, \cite{SaStSha}, \cite{Gu}, \cite {Zi}. It  was apparently pioneered  in the 1960's by a Bulgarian geophysicist D.~Zidarov \cite{Zi} and later mathematically developed by B.~Gustafsson  \cite{Gu}. Although a number of  interesting results about  motherbodies was obtained  in several special cases, \cite{SaStSha}, \cite{Gu}, \cite {Zi}  there is still   no consensus  about its appropriate general  definition. In particular,   no general existence and/or uniqueness results are  known at present. 

Below we use one  possible definition of a motherbody (measure) and study  a  natural  exterior inverse motherbody problem  in the complex plane. (In what follows we will only use Borel measures.)

\medskip 
\noindent 
{\it Main problem.}  Given a  germ $f(z)={a_0}/{z}+\sum_{i\ge 2}^\infty a_i/z^i,\; a_0\in \bR$ of an algebraic (or, more generally, analytic) function near  $\infty$,     is it possible to find a compactly supported  signed  measure whose Cauchy transform  coincides with (a branch of) the analytic  continuation of $f(z)$ a.e. in $\bC$? Additionally, for which $f(z)$  it is possible to find a positive measure with the above properties?

\medskip
If such a signed measure exists and its support is a finite union of compact semi-analytic curves and isolated points we will call it {\it a real motherbody measure} of the germ $f(z)$, see  Definition~\ref{def:main} below.  If the corresponding measure is positive then we call it {\it a positive motherbody measure} of $f(z)$. 

An obvious necessary condition for the existence of  a positive motherbody measure is  $a_0>0$ since $a_0$ is the total mass. A germ (branch) of an analytic function $f(z)={a_0}/{z}+\sum_{i\geq 2}^\infty a_i/z^i$ with $a_0>0$ will be called {\it positive}. If $a_0=1$ then $f(z)$ is called a {\it probability} branch. (Necessary and sufficient conditions for the existence of a probability branch of an algebraic function are given  in Lemma~\ref{lm:ProbBr} below.)

\medskip
The formal definition of a motherbody measure that we are using is as follows. 
\begin{definition}\label{def:main}\rm{ Given a germ $f(z)={a_0}/{z}+\sum_{i\ge 2}^\infty a_i/z^i,\; a_0>0$ of an analytic function near $\infty$, we say that a signed measure $\mu_f$ is its {\it motherbody} if:}

\noindent
\rm{(i)} its support $supp(\mu_f)$ is the union of finitely many points and finitely many compact semi-analytic curves in $\bC$;

\noindent
\rm{(ii)} The Cauchy transform of $\mu_f$ coincides in each connected component of the complement $\bC\setminus supp(\mu_f)$ with some branch of the analytic continuation of $f(z)$.


\end{definition}

\noindent
{\it Remark.} 
{\rm Notice that by Theorem 1 of \cite{BBB} if the Cauchy transform of a  positive measure  coincides   a.e. in $\bC$ with an algebraic function $f(z)$ then the support of this measure is a finite union of semi-analytic curves and isolated points. Therefore it is a motherbody measure according to  the above definition.  Whether the latter result extends to signed measures is unknown at present which  motivates the following question.}

\medskip 
\noindent 
{\it Problem 1.}    Is it true that if there exists a signed measure whose Cauchy transform satisfies an irreducible algebraic equation a.e. in $\bC$  then there exists, in general, another signed  measure whose Cauchy transform satisfies a.e. in $\bC$  the same algebraic equation and whose support is a finite union of compact curves and isolated points? Does there exist such a measure with a singularity on each connected component of its support? 
 

\medskip 
Classically the inverse problem in potential theory  deals with the question of  how  to restore a solid body or a (positive) measure from the information about  its potential near infinity.  The main efforts in this inverse problem  since the foundational paper of P.~S.~Novikov \cite{No} were  concentrated around the question about the uniqueness of a (solid) body with a given exterior potential, see e.g. recent \cite{GaSj1} and \cite{GaSj2} and the references therein.   P.~S.~Novikov  (whose main mathematical contributions are in the areas of mathematical logic and group theory) proved uniqueness for convex and star-shaped bodies with a common point of star-shapeness. The question about the uniqueness of contractible domains in $\bC$ with a given sequence of holomorphic moments was later posed anew by H.~S.~Shapiro, see  Problem 1, p. 193 of \cite{BOS} and answered negatively by  M.~Sakai in \cite{Sa}. A similar non-uniqueness example for non-convex plane polygons was reported by geophysicists in \cite{BroSt}, see also \cite{PaSh}. 

\medskip 
Observe that the knowledge of  the Cauchy transform of a compactly supported (complex-valued) measure $\mu$ near $\infty$  is equivalent to the knowledge of its (holomorphic) moment sequence $\{m_k^\mu\}, k=0,1,...$ where 
$$m_k^\mu=\int_\bC z^k d\mu(z).$$
(These moments 
  are important conserved quantities in the theory of Hele-Shaw flow,   \cite{KMWWZ}.) More precisely, if $$\mathfrak u_\mu(z):=\int_\bC \ln|z-\xi|d\mu(\xi)$$ is the logarithmic potential of  $\mu$ and 
$$\C_\mu(z):=\int_\bC \frac{d\mu(\xi)}{z-\xi}=\frac{\mathfrak \prt \mathfrak u_\mu(z)}{\prt z}$$ is its Cauchy transform then the Taylor expansion of $\C_\mu(z)$  at $\infty$ has the form: 
$$\C_\mu(z)=\frac{m_0(\mu)}{z}+\frac{m_1(\mu)}{z^2}+\frac{m_2(\mu)}{z^3}+\ldots .$$

Observe also the important relations: 
$$\C_\mu=\frac{\prt u_\mu}{\prt z} \quad \text{and} \quad  
\mu=\frac{1}{\pi}\frac{\prt \C_\mu}{\prt \bar z}=\frac{1}{\pi}\frac{\prt^2 u_\mu}{\prt z \prt \bar z}$$
which should be understood as equalities of distributions. 

   \medskip 
It turns out that  the existence of a compactly supported positive measure with a given Cauchy transform $f(z)$ near $\infty$  imposes no conditions on a germ except for the obvious $a_0>0$, see Theorem~\ref{th:prop} below.  

On the other hand, the requirement that the Cauchy transform coincides with (the analytic continuation) of a given germ $f(z)$ a.e. in $\bC$  often leads to   additional restrictions on the germ $f(z)$ which are not  easy to describe in terms of the defining algebraic equation, see \S~\ref{sq-roots}.  

Below we study two  classes of algebraic functions  of very different origin and our results for these two cases are very different as well. For the first class the obvious necessary condition $a_0=1$ seems to be  sufficient for the existence of a positive motherbody measure. (At present we can prove this fact only under certain additional restrictions.) For the second class the set of admissible germs has a quite complicated structure. These results together with a number of conjectures  seem to indicate that  it is  quite difficult, in general, to answer when a given algebraic germ $f(z)$ admits  a motherbody measure and if it does then how many. 

Several concluding  remarks are in place here.  
Our interest in  probability measures  whose  Cauchy transforms are algebraic functions a.e. in $\bC$  was sparked by the pioneering work \cite{BR}. Since then the class of interesting examples where such situation occurs has been substantially broadened, see  \cite{BBS}, \cite{HS1}, \cite{STT}.  
Some general local results when one considers a collection of locally analytic functions instead of a global algebraic function were obtained in \cite{BB} and later extended in  \cite {BBB}.  

\medskip
\noindent 
{\bf Acknowledgements.}  The authors are obliged to   T.~Bergkvist, J.~E.~Bj\"ork, late J.~Borcea, B.~Gustafsson,  A.~Mart\'inez-Finkelshtein, and H.~Rullg\aa rd for numerous discussions of the topic over the years. Studies of quadrature domains by H.~S.~Shapiro were in many ways an important source of inspiration for the present project.  It is our pleasure  to acknowledge this influence.





\section{Some general facts}\label{gen:facts}

The first  rather simple result of the present paper  (which apparently is  known to the specialists)  is as follows. 

\begin{theorem}\label{th:prop} Given an arbitrary germ $f(z)= {a_0}/{z}+\sum_{i\ge 2}^\infty a_i/z^i,\; a_0>0$ of an  analytic function near  $\infty$ there exist   (families of) positive compactly supported in $\bC$ measures  whose Cauchy transform   near $\infty$ coincides with $f(z)$. 
\end{theorem}

\begin{proof}[Proof of Theorem~\ref{th:prop}] Given a branch $f(z)=a_0/z+\sum_{i\ge 2}a_i/z^i$ of an analytic function near $\infty$ we first take a germ of its 'logarithmic potential',   i.e. a germ  $h(z)$ of  harmonic function such that  $h(z)=a_0\log |z| +\dots$ satisfying the relation 
$\partial h/\partial z=f(z)$ in a punctured neighborhood of $\infty$ where $\dots$ stands for a germ of harmonic function near $\infty$.   For any sufficiently large positive $v$ the  connected component $\gamma_v$ of the level curve $ h(z)=v$  near infinity is  closed and simple. It makes one turn around $\infty$. To get  a required positive measure whose Cauchy transform coincides with $f(z)$ near $\infty$ take  $\gamma_v$ for any $v$ large enough and consider the complement $\Omega_v=\bC P^1\setminus Int(\gamma_v)$ where  $Int(\gamma_v)$ is the interior of $\gamma_v$ containing $\infty$. Consider the equilibrium measure of mass $a_0$ supported on $\Omega_v$. By Frostman's theorem,   this measure is in fact supported on $\gamma_v$ (since $\gamma_v$ has no polar subsets), its potential is constant on $\gamma_v$ and this potential  is harmonic in the complement to the support. (For definition and properties of equilibrium measures as well as Frostman's theorem\footnote{Otto Frostman has spent a  substantial part of his professional life at the same department in Stockholm where we are currently employed.} consult  \cite{Rans}.) Thus it should coincide with $h(z)$ in $Int(\gamma_v)$ since the total mass is correctly chosen. Then its Cauchy transform coincides with $f(z)$ in $Int(\gamma_v)$.    
\end{proof}
 
\begin{example} {\rm If we choose  $f(z)=1/z$ as the branch at $\infty$, then $h(z)=\log \vert z\vert $ and $\gamma_v$ is the circle $\vert z\vert = e^v$. The equilibrium measure in this case is the uniform probability  measure  on this circle, and its Cauchy transform $C(z)$ will be $0$ inside the circle and $1/z$ outside. Since $0$ is not the analytic continuation of $1/z$, the uniform measure on the circle is not a motherbody measure. However the unit point mass $\mu$ at the origin, will have Cauchy transform $C(\mu)=1/z$ in $\bC\setminus 0$ and therefore will be  a motherbody measure for the germ $1/z$ at $\infty$.}
\end{example}

We now give a necessary condition and a slightly stronger sufficient condition for an algebraic curve given by the equation 
\begin{equation}\label{eq:main}
P(\C,z)=\sum_{(i,j)\in S(P)}\al_{i,j}\C^iz^j=0
\end{equation}  
to have a probability branch at $\infty$. Here every $\al_{i,j}\neq 0$ and $S(P)$ is an arbitrary finite subset of pairs of non-negative integers, i.e an arbitrary set of monomials in $2$ variables. In other words,   $S(P)$ is the set of all  monomials appearing in $P$ with non-vanishing coefficients. 
(In what follows $\C$  stands for the variable corresponding to the Cauchy  transform.) 

\begin{lemma}\label{lm:ProbBr} If the algebraic curve given by  equation~\eqref{eq:main}  has a probability branch at $\infty$ then
\begin{equation}
\label{eq:diag}
\sum_{i} \al_{i,i-M(P)}=0\quad\text{ where}\quad M(P)=\min_{(i,j)\in S(P)} i-j.
\end{equation}

In particular, there should be at least two distinct monomials in $S(P)$ whose difference of indices $i-j$ equals $M(P)$.

If equation \eqref{eq:diag} is satisfied, and additionally 
\begin{equation}
\label{eq:branchpoint}\sum_{i}i \al_{i,i-M(P)}\neq 0,
\end{equation}
then there is a unique probability branch at $\infty$ satisfying equation~\eqref{eq:main}.

\end{lemma}

\begin{proof}
Substituting $w=1/z$ in \eqref{eq:main} we get 
$P(\C,w)=\sum_{(i,j)\in S(P)}\al_{i,j}\C^i/w^j.$ Assuming that the algebraic curve given by $P(\C,w)=0$ has a branch 
$$\C=w+\sum_{l=2}^\infty a_lw^l$$
 where $a_l,\;l=2,3,...$ are undetermined coefficients, we substitute the latter expression in the above equation and get 
\begin{equation} \label{eq:rec}
\sum_{(i,j)\in S(P)}\al_{i,j}(w+\sum_{l=2}^\infty a_lw^l)^i/w^j=0.
\end{equation}
Collecting the entries of the minimal degree which equals $M(P)=\min_{(i,j)\in S(P)}i-j$,  we obtain an obvious necessary condition for the solvability of \eqref{eq:rec}   given by 
\begin{equation}
\label{eq: firststep}
\sum\al_{i, i-M(P)}=0.
\end{equation}

Let us now show that   \eqref{eq:diag}, together with \eqref{eq:branchpoint},  are sufficient for the solvability of \eqref{eq:rec} with respect to the sequence of coefficients $a_2, a_3,\dots$. 
Indeed, due to algebraicity of  $P(\C,z)$ it suffices to prove the formal solvability of \eqref{eq:rec}. Let $\C=\D w$ and rewrite the equation as 
$$
\tilde P(\D,w)=\sum_{(i,j)\in S(P)}\al_{i,j}\D^iw^{i-j-M(P)}=0,
$$
which is now a polynomial in $w$ and $\D$. Assume that $d_r=1+a_2w+...+a_{r+1}w^r$ satisfies 
\begin{equation}
\label{eq:inductionstep}
\tilde P(d_r,w)\equiv 0\ \text{mod } z^{r+1}.
\end{equation}
The fact that this equation holds  for $d_0=1$ is exactly the relation (\ref{eq: firststep}) which gives the basis of the inductive construction of $d_1,d_2,...$, as follows. Letting $\tilde P'_1(\C,w)$ be the partial derivative of $\tilde P(\C,w)$ with respect to the first variable, we have the following relation for the undetermined coefficient $a_{r+2}$:
\begin{eqnarray*}
\tilde P(d_r+a_{r+2}w^{r+1},w)\equiv \tilde P(d_r,w)+\tilde P'_1(d_r,w)a_{r+2}w^{r+1}\\
\equiv
\tilde P(d_r,w)+\tilde P'_1(1,0)a_{r+2}w^{r+1} \text{  mod } w^{r+2}.
\end{eqnarray*}
Since we have assumed that $\tilde P'_1(1,0)=\sum_{i}i \al_{i,i-M(P)}$ is non-zero, and that, by the induction assumption,  $\tilde P(d_r,w)\equiv bw^{r+1} \text{  mod } w^{r+2},\ b\in \mathbb C$,  we can solve the latter equation for $a_{r+2}$. Thus, by induction, we have proven that there is a formal series solution of \eqref{eq:rec}. Therefore  conditions \eqref{eq:diag}  and  \eqref{eq:branchpoint} are sufficient for the existence of a probability branch which is also unique in this case. 
\end{proof}

\noindent
{\it Remark.} 
Note that for an irreducible algebraic curve defined by \eqref{eq:main} the second condition  (\ref{eq:branchpoint})  in the lemma says that $z=\infty$ is not its branching point.  It is clearly possible, though cumbersome, to give necessary and sufficient  conditions for the existence of a probability branch in terms of algebraic relations between the coefficients of the equation.  An example of an equation that does not satisfy both  conditions in the lemma, but still has a probability branch as a solution,  is $P(\C,z)=(\C z-1)^2$. 

\begin{figure}

\begin{center}
\includegraphics[scale=0.4]{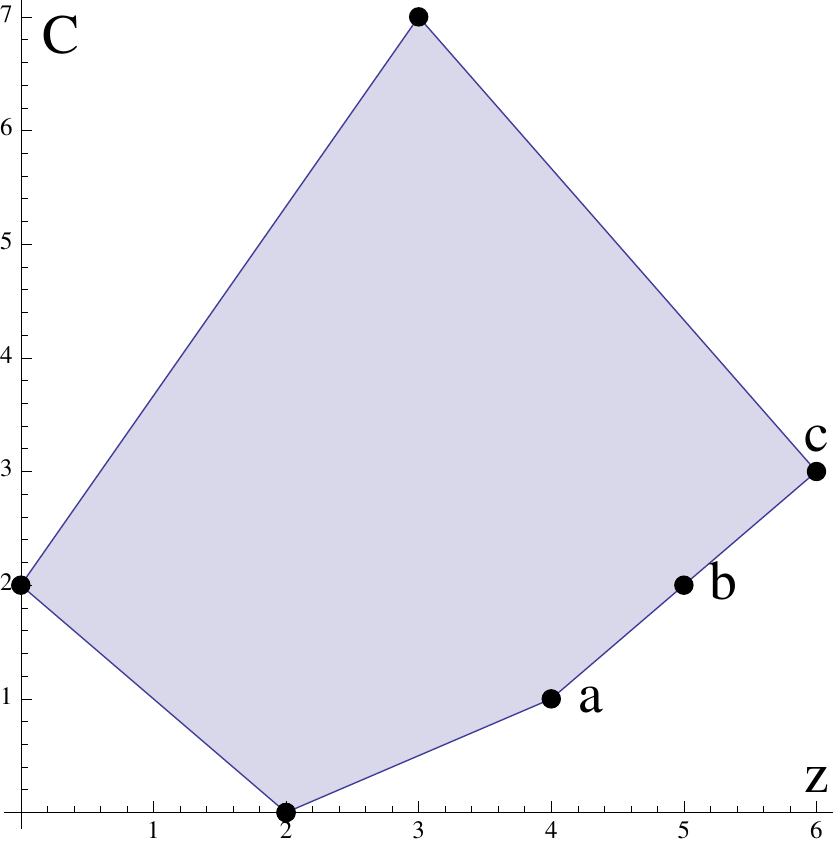} \quad \quad \includegraphics [scale=0.4]{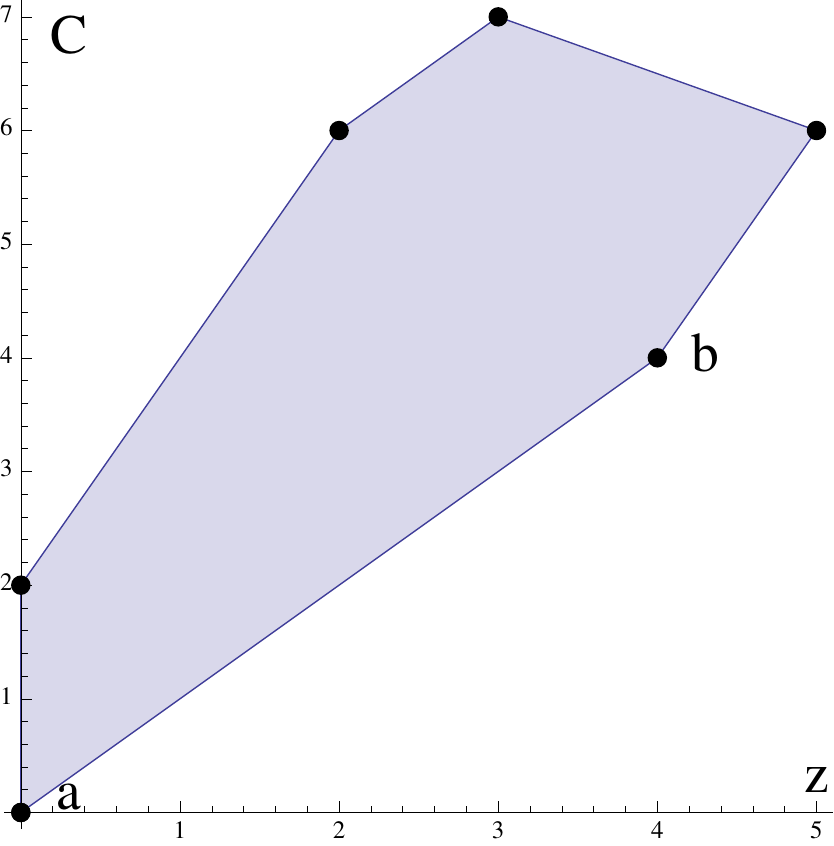}  \quad \quad \includegraphics [scale=0.4]{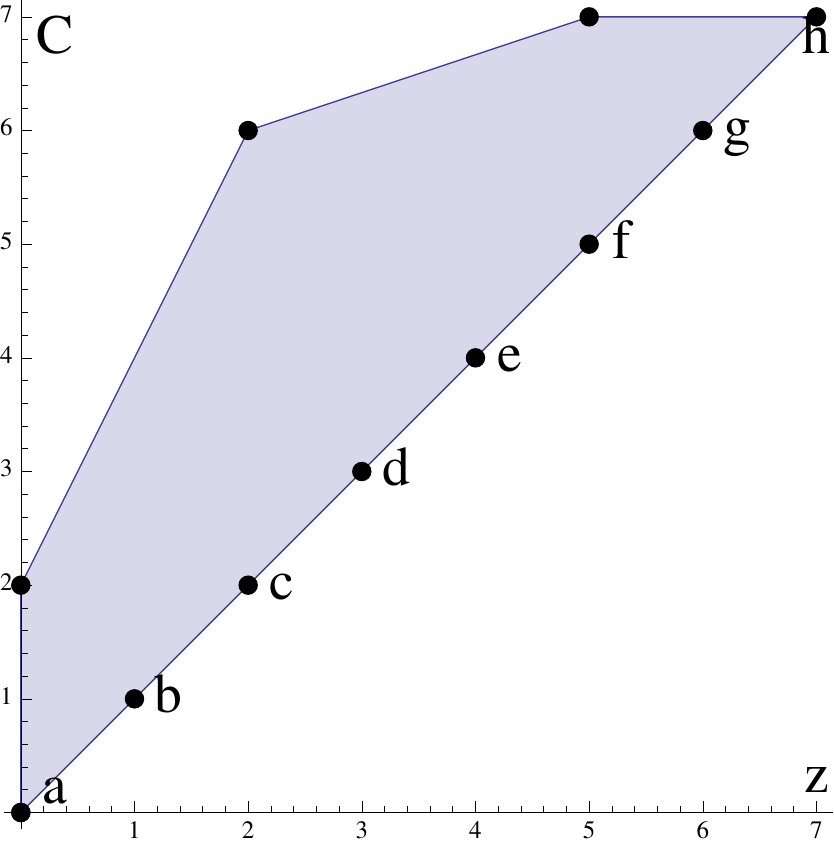}
\end{center}


\caption{Three examples of  Newton polygons.}
\label{fig1}
\end{figure}

In Fig.~1   the  necessary condition of Lemma~\ref{lm:ProbBr} for the polynomials with Newton polygons shown on each of three pictures is   that the sum of the coefficients at the monomials labelled by the letters a, b,... vanishes. 

Since we are working with irreducible algebraic curves (obtained as the analytic continuation of  given branches at $\infty$)  we will need the following  statement. Let $S$ be an arbitrary set of monomials in $2$ variables. Denote by  $Pol_S$ the linear span of these monomials with complex coefficients, and denote by $\N_S$  the {\it Newton polygon} of $S$, i.e. its convex hull.   

\begin{lemma}\label{lm:irred}  A generic polynomial in  $Pol_S$ is irreducible if and only if: 

\medskip
\noindent
\rm{(i)} $\N_S$ is two-dimensional, i.e. not all monomials in $S$ lie on the same affine line;
 
\medskip
 \noindent
\rm{(ii)} $S$ contains a vertex on both coordinate axes, i.e. pure powers of both variables.  
\end{lemma}

Notice that (ii) is already satisfied if $S$ contains the origin (i.e. polynomials in $Pol_S$ might have a non-vanishing constant term).

\begin{proof}
Observe  that the property that a generic polynomial from $Pol_S$ is irreducible is inherited, i.e. if $S$ contains a proper subset $S'$ with the same property then it automatically holds for $S$. The necessity of both conditions (i) and (ii) is obvious. If (i) is violated then  any polynomial in $Pol_S$ can be represented as a polynomial in one variable after an appropriate change of variables. If (ii) is violated that any polynomial in $Pol_S$ is divisible by a (power of a) variable.  

To prove sufficiency we have to consider several cases. If $S$ contains $\C^i$ and $z^j$ where both $i$ and $j$ are positive then already a generic curve of the form $\al\C^i+\be \C^jz^k +\ga z^l=0$ is irreducible unless $\C^i$,  $\C^jz^k$ and $z^l$ lie on the same line. If $S$ contains $\C^i$ and $1$ for some positive $i$ and no other pure powers of $z$ (or, similarly, $z^j$ and $1$ for some positive $k$ and no other pure powers of $\C$) then $S$ should contain another monomial $\C^lz^k$ with $l$ and $k$ positive.  But then a generic curve of the form $A+B\C^i+C\C^lz^k=0$ is irreducible. Finally, if the only monomial on the axes is $1$ but there exists two monomials $\C^{i_1}z^{j_1}$  and $\C^{i_2} z^{j_2}$ with $i_1/j_1\neq i_2/j_2$ then  
a generic curve of the form $A+B\C^{i_1}z^{j_1}+C\C^{i_2}z^{j_2}=0$ is irreducible. 
\end{proof}

\begin{corollary}\label{cor:irredpos} An irreducible polynomial $P(\C,z)$ having a probability branch has a non-negative   $M(P)=\min_{(i,j)\in S(P)} i-j$. If $M(P)=0$ then the set $S(P)$ of monomials of $P(\C,z)$ must contain the origin, i.e.,  a non-vanishing constant term. 
\end{corollary}
 
If we denote by $\N_P=\N_{S(P)}$ the Newton polygon of $P(\C,z)$ then the  geometric interpretation of the latter corollary is that $\N_P$ should (non-strictly) intersect the bisector of the first quadrant in the plane $(z,\C)$.

\medskip
The case $M(P)=0$ will be of special interest to us, see \S~\ref{minalgfun}. We say that  a polynomial $P(\C,z)$ (and the  
algebraic function given by $P(\C,z)=0$)  is {\em balanced} if it satisfies the condition $M(P)=0$. 
For polynomials with $M(P)>0$   having a positive branch  the problem of existence of a motherbody measure seems to be dramatically more complicated than for $M(P)=0$, see \S~\ref{sq-roots}. 

For $M(P)>0$ a rather simple situation occurs when $\C$ is a rational function, i.e. $\C$ satisfies a linear equation.  

\begin{lemma}\label{lm:rational}
A (germ at $\infty$ of a) rational function $\C(z)=\frac{z^n+...}{z^{n+1}+...}$ with coprime numerator and denominator admits a motherbody measure if and only if it has all simple poles with real residues. If all residues are positive the corresponding motherbody measure is positive. 
\end{lemma}

\begin{proof}
Recall the classical relation 
$$\mu= \frac{1}{\pi}\frac{\partial{\C_\mu(z)}}{\partial \bar z}$$
between a measure $\mu$ and its Cauchy transform $\C_\mu$ 
where the derivative is taken in the sense of distributions. By assumption $\C_\mu(z)$ coincides almost everywhere with a given rational function.  Therefore, by the above formula, $\mu$ is the measure supported at the poles of the rational function that gives a weight to each pole coinciding up to a factor $\pi$ with the residue at that pole. Since by assumption $\mu$ has to be a real measure this implies that all poles of the rational function should be simple and with real residues. 
\end{proof}

The above statement implies that the set of rational function of degree $n$ admitting a motherbody measure has dimension equal to half of the dimension of the space of all rational functions of degree $n$. The case when $\C_\mu$ satisfies a quadratic equation  is considered in some detail in \S~\ref{sq-roots}.  





\section {Balanced algebraic functions} \label{minalgfun}

In the  case of balanced algebraic functions our main conjecture is as follows. 

\begin{conjecture}\label{conj:M=0}
An arbitrary balanced irreducible polynom $P(\C,z)$ with a positive branch  admits a positive motherbody measure. \end{conjecture}

Appropriate Newton polygons are shown on the central and the right-most pictures  in Fig.~1.  
 The next result essentially proven in \cite{BBS} strongly supports the above conjecture. Further supporting statements can be found in \cite{Ber}. 
  
 \begin{theorem}\label{th:main}  An arbitrary balanced irreducible polynomial $P(\C,z)$ with a unique probability branch which additionally satisfies the following requirements:  
 
 \medskip
 \noindent 
 {\rm (i)} $S(P)$ contains a diagonal monomial $\C^nz^n$ which is lexicographically bigger than any other monomial in $S(P)$;
 
 \noindent 
{\rm (ii)}  the probability branch is the only positive branch of  $P(\C,z)$;  
 
 \medskip
 \noindent 
 admits a  probability motherbody measure.   \end{theorem}

By 'lexicographically bigger' we mean that the pair $(n,n)$ is coordinate-wise not smaller than any other pair of indices in $S(P)$ as shown on the right-most picture in Fig.~1.  Condition \rm(i) is the only essential restriction in Theorem~\ref{th:main} compared to Conjecture~\ref{conj:M=0} since condition \rm(ii) is generically satisfied.  Observe also that an irreducible balanced polynomial  $P(\C,z)$ must necessarily have a non-vanishing constant term, see Lemma~\ref{lm:irred}.  Balanced polynomials satisfying assumptions of Theorem~\ref{th:main} are called 
{\em excellent balanced}   polynomials.  

\medskip 
An interesting detail about the latter theorem is that its proof has hardly anything to do with potential theory. 
Our proof of Theorem~\ref{th:main}  is based on certain properties of  eigenpolynomials of the so-called exactly solvable differential operators, see e.g. \cite {BBS}. We will construct the required motherbody measure as the weak limit of a sequence of root-counting measures of these  eigenpolynomials. 

\begin{definition} {\rm A linear ordinary differential operator 
\begin{equation}\label{eq:oper}
\dq=\sum_{i=1}^kQ_i(z)\frac{d^i}{dz^i}
\end{equation}
 where each $Q_i(z)$ is a polynomial of degree at most $i$ and there exists at least one value  $i_0$ such that 
$\deg Q_{i_0}(z)=i_0$ is called {\em exactly solvable}. An exactly solvable operator is called {\em non-degenerate} if $\deg Q_{k}(z)=k$.  The symbol $T_\dq$ of the operator \eqref{eq:oper} is, by definition, the bivariate polynomial
 \begin{equation}\label{eq:symbol}
T_\dq(\C,z)=\sum_{i=1}^kQ_i(z)\C^i.
\end{equation}

}
\end{definition}

Observe that $\dq$ is exactly solvable if and only if $T_\dq$ is balanced. (Notice that $T_\dq$, by definition, has no constant term.)

\medskip
Given an exactly solvable $\dq=\sum_{i=1}^kQ_i(z)\frac{d^i}{dz^i}$,  consider the following {\em homogenized spectral problem}: 
\begin{equation}\label{eq:spec}
Q_k(z)p^{(k)}+\la Q_{k-1}(z)p^{(k-1)}+\la^2Q_{k-2}(z)p^{(k-2)}+...+\la^{k-1}Q_1(z)p'=\la^kp,
\end{equation}
where $\la$ is called the {\em homogenized spectral parameter}. Given a positive integer $n,$ we want to determine  all values $\la_n$ of the spectral parameter such that equation \eqref{eq:spec} has a polynomial solution $p_n(z)$ of degree $n$.

\medskip
Using   notation  $Q_i(z)=a_{i,i}z^i+a_{i,i-1}z^{i-1}+...+a_{i,0},\; i=1,...,k,$ one can easily check that  the eigenvalues  $\la_n$  satisfy the equation: 
\begin{equation} \label{eq:ind}
a_{k,k}n(n-1)...(n-k+1)+\la_n a_{k-1,k-1}n(n-1)...(n-k+2)+...+\la_n^{k-1}na_{1,1}=\la_n^k.
\end{equation} 
(If $\dq$ is non-degenerate, i.e. $a_{k,k}\neq 0$, then there typically  will be $k$ distinct values of $\la_n$ for $n$ large enough.)

Introducing the normalized eigenvalues $\widetilde \la_n=\la_n/n,$ we get that $\widetilde \la_n$ satisfy  the equation: 
$$a_{k,k}\frac{n(n-1)...(n-k+1)}{n^k}+\widetilde{\la}_n a_{k-1,k-1}\frac{n(n-1)...(n-k+2)}{n^{k-1}}+...+\widetilde{\la}_n^{k-1}a_{1,1}=\widetilde{\la}_n^k.$$
If $\widetilde \la=\lim_{n\to \infty}\widetilde \la_n$ exists then it should satisfy the relation:  
\begin{equation}\label{eq:stab}
a_{k,k}+ a_{k-1,k-1}\widetilde\la+...+a_{1,1}\widetilde\la^{k-1}=\widetilde \la^k.
\end{equation} 

If $\dq$ is a degenerate exactly solvable operator, then let $j_0$ be the maximal $i$ such that  $\deg Q_i(z)=i$. By definition,  $a_{j,j} $ vanish for all $j> j_0$.  Thus,   the first $k-j_0$ terms in \eqref{eq:ind} vanish as well implying that  \eqref{eq:ind} has $j_0$ non-vanishing and $(k-j_0)$ vanishing roots. 

\begin{lemma} \label{lm:sim}  Given an exactly solvable operator $\dq$ as above,  

\noindent
\begin{itemize}
\item [(i)] For the sequence of vanishing 'eigenvalues'  $\la_n=0$ there exists a finite upper bound of the degree of a non-vanishing eigenpolynomial $\{p_n(z)\}$. 

\noindent
\item [(ii)] For any sequence $\{\la_n\}$ of eigenvalues such that $\lim_{n\to\infty}\frac{\la_n}{n}=\widetilde\la_l$ where $\widetilde\la_l$ is some non-vanishing simple root of \eqref{eq:stab} the sequence of their eigenpolynomials is well-defined for all $n>N_0$, i.e. for all sufficiently large $n$.

\end{itemize} 
\end{lemma} 

\begin{proof} Item (i) is obvious since \eqref{eq:spec} reduces to $Q_k(z)p^{(k)}=0$ when $\la=0$ implying $p^{(k)}=0$ which is impossible for any polynomial of degree exceeding $k-1$. 

To prove (ii), notice that \eqref{eq:ind} defines $\widetilde\la$ as an algebraic function of the (complex) variable $n$, and that \eqref{eq:stab} is then the equation for the value at $n=\infty$. By assumption there is a unique branch with the value $\widetilde\la_l$, and then the corresponding branch of the algebraic function will intersect other branches in at most a finite number of points. Hence for $n$ very large we may identify  $\widetilde\la_l(n)$ as belonging to this branch.
\end{proof} 

The next result  is  central in our consideration.  

\begin{theorem} \label{th:loc}{{\rm [}see Theorem~2, \cite{BBS}{\rm ]}} For any non-degenerate exactly solvable operator $\dq$, such that \eqref{eq:stab} has no double roots, there exists $N_0$ such that the roots of all eigenpolynomials of the homogenized problem   \eqref{eq:spec} whose degree exceeds $N_0$ are bounded, i.e. there exists a disk centered at the origin containing all of them at once. 
\end{theorem} 

Unfortunately the existing proof of the latter theorem is too long and technical to be reproduced in the present paper. 
 The next local result is a keystone in  the proof of Theorem~\ref{th:main}, comp. Proposition~3 of \cite {BBS}.  

\begin{theorem} \label{th:CT} Given an exactly solvable $\dq=\sum_{i=1}^kQ_i(z)\frac{d^i}{dz^i},$ consider  a sequence $\{p_n(z)\}, \deg p_n(z)=n$ of the eigenpolynomials of  \eqref{eq:spec}  such that  the sequence $\{\la_n\}$ of their  eigenvalues satisfies the condition $\lim_{n\to\infty}\frac{\la_n}{n}=\widetilde\la_l$ where $\widetilde\la_l$ is some non-vanishing root of \eqref{eq:stab}.  Let $L_n(z)=\frac{p_n'(z)}{\la_np_n(z)}$ be the normalized logarithmic derivative of $p_n(z)$. If  the sequence $\{L_n(z)\}$  converges  to a function $L(z)$  in some open domain $\Omega \subset \bC$, and the derivatives of $L_n(z)$ up to  order $k$ are uniformly bounded in $\Omega,$  then $L(z)$ satisfies in $\Omega$ the  symbol equation: 
\begin{equation}\label{algfunc}
Q_k(z)L^k(z)+Q_{k-1}(z)L^{k-1}(z)+...+Q_1(z)L(z)=1.
\end{equation} \end{theorem} 
\begin{proof} 
Note that each $L_{n}(z)=\frac{p_{n}'(z)}{\la_{n}p_{n}(z)}$
is well-defined and analytic in any open domain $\Omega$  free from the zeros of
$p_{n}(z)$.
             Choosing such a  domain $\Omega$ and an appropriate branch of the
logarithm such that $\log p_{n}(z)$ is defined in $\Omega$,  consider a
primitive function
$M(z)=\lambda^{-1}_{n}\log p_{n}(z)$ which is also well-defined and
analytic in $\Omega$.

Straightforward calculations give: $e^{\lambda_{n}M(z)}=p_{n}(z),\;$
$p'_{n}(z)=p_{n}(z)\lambda_{n}
L_{n}(z)$, and
$p''_{n}(z)=p_{n}(z)(\lambda^2_{n}L^2_{n}(z)+\lambda_{n}L_{n}'(z))$.
More generally,
$$
\frac{d^i}{dz^i}(p_{n}(z))p_{n}(z)\left(\lambda^i_{n}L^i_{n}(z)+\lambda^{i-1}_{n}F_{i}(L_{n}(z),L'_{n}(z),
\ldots,
L^{(i-1)}_{n}(z))
\right),$$
where the second term
\begin{equation}
	\lambda_{n}^{i-1}F_{i}(L_n,L_n',\ldots,L_n^{(i-1)})
	\label{eq:Fdesc}
\end{equation}
is a polynomial of degree $i-1$ in $\lambda_{n}$. 
The equation $\dq_{\lambda_{n}}p_{n}(z)=0$ gives us: 
$$
p_{n}(z)
\left(\sum_{i=0}^{k}Q_{i}(z)\la_{n}^{k-i}\left(\lambda_{n}^iL_{n}^{i}(z)+\lambda_{
n}^{i-1}F_
{i}(L_{n}(z),L'_{n}(z),\ldots,L^{(i-1)}_{n}(z))\right)\right)=0
$$
or, equivalently,
\begin{equation}
	\label{Dlog}
	\lambda_{n}^k\sum_{i=0}^{k}Q_{i}(z)\left(L^i_{n}(z)+\lambda_{n}^{-1}F_{i}(L_{n}(z),L
'_{n}(z),
	\ldots,L^{(i-1)}_{n}(z))\right)=0.
\end{equation}
Letting $n\to \infty$ and using  the
boundedness assumption for the first $k-1$ derivatives we get the required
equation (\ref{algfunc}).
\end{proof}

\begin{proof}[Sketch of Proof of Theorem~\ref{th:main}]
Take an excellent balanced irreducible polynomial $P(\C,z)=\sum_{i=1}^kQ_i(z)\C^i$ having a probability branch at $\infty$. Since $P(\C,z)$ has a non-vanishing constant term, we can without loss of generality  assume  that it is equal to $-1$. Consider its corresponding  differential operator $\mathfrak{d}_P(z)=\sum_{i=1}^kQ_i(z)\frac{d^i}{dz^i}$, that is the operator  whose symbol $T_\dq$ equals  $P(\C, z)+1$.  Notice that since $P(\C,z)$ has a probability branch at $\infty$ then there exists a root of   \eqref{eq:stab} which is equal to $1$. Therefore, there exists a sequence $\{\la_n\}$ of eigenvalues of $\mathfrak{d}_P(z)$ satisfying the condition $\lim_{n\to\infty}\frac {\la_n}{n}=1$. 

 To settle  Theorem~\ref{th:main}, notice that by Theorem~\ref{th:loc} for any $\dq_P(z)$ as above the roots of all its eigenpolynomials of all sufficiently large degrees are localized.  Choose a sequence  $\{\la_n\}$ of eigenvalues of $\mathfrak{d}_P(z)$ satisfying the condition $\lim_{n\to\infty}\frac {\la_n}{n}=1$ which exists for any excellent $P(\C,z)$.  Consider the corresponding sequence $\{p_n(z)\}$ of eigenpolynomials of $\mathfrak{d}_P(z)$, and the sequence $\{\mu_n\}$ of the  root-counting measures of these eigenpolynomials, together with  the sequence $\{\C_n(z)\}$ where $\C_n(z)=\frac{p_n(z)}{np_n(z)}$ are their Cauchy transforms.  (Observe that the Cauchy transform of a polynomial $P$ of degree $n$ equals $\frac{P^\prime}{nP}$.) By Theorem~\ref{th:loc} all zeros of all these $p_n(z)$ are contained in some fixed disk. Therefore the supports of all $\mu_n$ are bounded implying that there exists a weakly converging  subsequences 
$\{\mu_{i_n}\}$ of the sequence $\{\mu_n\}$. Chosing an appropriate further  subsequence we can guarantee that the corresponding subsequence $\{\C_{i_n}(z)\}$ converges almost everywhere in $\bC$.  (The rigorous argument for the latter claim is rather technical, see details in \cite {BBS}.)  Taking a further subsequence,  if necessary, we get that the subsequences $\{\C'_{i_n}(z)\},  \{\C_{i_n}^{\prime\prime}(z)\},..., \{\C_{i_n}^{(k)}(z)\}$ will be  bounded a.e. on any compact subset of $\bC$.   Finally, notice that if the limit $\lim_{n\to\infty}\C_{i_n}(z)=\C(z)$ exists in some  domain $\Omega$ then 
also $\lim_{n\to\infty}L_{i_n}(z)$ exists in $\Omega$ and equals $\C(z)$ where $L_n(z)=\frac{p_n(z)}{\la_n p_n(z)}$ and  the sequence $\{\la_n\}$ is chosen as above.   Since it  satisfies the condition $\lim_{n\to\infty}\frac {\la_n}{n}=1,$  we have settled the existence of a probability measure  whose Cauchy transform coincides a.e. in $\bC$ with a branch of the algebraic function given by the equation $\sum_{i=1}^kQ_i(z)\C^i=1$ which is the same as  $P(\C,z)=0$. 
\end{proof}


\section {Cauchy transforms satisfying quadratic equations and quadratic differentials} \label{sq-roots}

In this section  we discuss  which quadratic equations of the form: 

\begin{equation}\label{quadr}
P(z)\C^2+Q(z)\C+R(z)=0, 
\end{equation}
with $\deg P=n+2,\;\deg Q\le n+1,\;\deg R\le n$ 
admit  motherbody measure(s). 
 
For the subclass of \eqref{quadr} with $Q(z)$ identically vanishing, such results  were in large obtained in 
 \cite{STT}. Very close statements were independently and  simultaneously obtained in \cite{MFR1} and \cite{MFR2}. To go further,  we need to recall some basic facts about quadratic differentials.

\subsection{Basics on quadratic differentials} 
A (meromorphic) quadratic differential $\Psi$ on a compact  orientable Riemann surface $Y$ without boundary is
a (meromorphic) section of the tensor square
$(T^*_{\mathbb{C}}Y)^{\otimes2}$ of the holomorphic cotangent
bundle $T^*_{\mathbb{C}}Y$. The zeros and the poles of $\Psi$
constitute the set of \emph{singular points} of $\Psi$ denoted by
$Sing_{\Psi}$. (Non-singular points of $\Psi$ are usually called {\em
regular}.)

If $\Psi$ is locally represented in two intersecting charts by $h(z)dz^2$ and by
$\tilde{h}(\tilde{z})d\tilde{z}^2$  resp. with a
transition function $\tilde{z}(z)$, then 
$h(z)=\tilde{h}(\tilde{z})\left({d\tilde{z}}/{dz}\right)^2.$
Any quadratic differential induces a canonical metric on its Riemann
surface $Y$, whose length element  in local coordinates is given by
$$|dw|=|h(z)|^{\frac{1}{2}}|dz|.$$

The above canonical  metric $|dw|=|h(z)|^{\frac{1}{2}}|dz|$ on $Y$ is closely related to  two distinguished line fields  given by the condition that
$h(z)dz^2>0$ and $h(z)dz^2<0$ resp.    
The integral curves of the first field are called \emph{horizontal
trajectories} of $\Psi$, while the  integral curves  of the second field are called \emph{vertical trajectories} of $\Psi$. In what follows we
will mostly use horizontal trajectories of 
quadratic differentials  and reserve the term 
\emph{trajectories} for the horizontal ones. 

Here  we only consider rational quadratic differentials, i.e. $Y=\bCP^1$.  Any such  quadratic differential $\Psi$ will be
globally given  in $\bC$ by $\phi(z)dz^2$, where $\phi(z)$ is a complex-valued rational function. 

Trajectories of $\Psi$ can be naturally  parameterized by their 
arclength. In fact, in a neighborhood of a regular point $z_0$ on
$\bC$, 
 one can introduce a local coordinate called a {\em canonical parameter}
and  given by
\[w(z)
:=\int_{z_0}^{z}\sqrt{\phi(\xi)}d\xi.\]
One can easily check that 
$dw^2={\phi(z)}dz^2$ 
implying that  horizontal trajectories in the $z$-plane correspond to 
horizontal straight lines in the $w$-plane, i.e they are defined by the condition $\Im\, w=const$.

A trajectory of a meromorphic quadratic differential $\Psi$ given on a compact $Y$ without boundary is called  \emph{singular} if there exists a singular point of $\Psi$ belonging to its closure. 
For a given quadratic differential $\Psi$ on  a compact surface  $Y,$  denote by 
$K_\Psi\subset Y$ the union of all its singular trajectories and singular points. In general, $K_\Psi$ can be very complicated. In particular, it can be dense in some subdomains of $Y$. 

We say that a singular trajectory is {\it double singular} if   a) it approaches  a singular point  in both directions, i.e. if one moves  from a given point on this trajectory in one direction  or the opposite; b) each of these singular points is either a zero or a simple pole. 
We denote by $DK_\Psi\subseteq K_\Psi$ (the closure of) the set of double singular trajectories of \eqref{eq:maindiff}. (One can easily show that $DK_\Psi$ is an imbedded (multi)graph in $Y$. Here by a {\it multigraph} on a surface we mean a graph with possibly multiple edges and loops.) Finally, denote by $DK^0_\Psi\subseteq DK_\Psi$ the subgraph of $DK_\Psi$ consisting of (the closure of) the set of  double singular trajectories whose both ends are zeros of $\Psi$.

A non-singular trajectory $\gamma_{z_0}(t)$ of a meromorphic $\Psi$  is called \emph{closed} if $\exists \ T>0$ such that
$\gamma_{z_0}(t+T)=\gamma_{z_0}(t)$ for all $t\in\mathbb{R}$. The
least such $T$ is called the \emph{period} of $\gamma_{z_0}$.  
A quadratic differential $\Psi$ on a compact Riemann surface $Y$ without boundary is called 
\emph{Strebel}  if the set of its closed trajectories covers $Y$ up to a set of Lebesgue measure 
 zero.

\subsection{General results on Cauchy transforms satisfying quadratic equations}

In this subsection we find the relation of the question for which triples of polynomials $(P,Q,R)$ equation~\eqref{quadr} admits a real motherbody measure to a certain problem about rational quadratic differentials. We start with the following necessary condition. 

\begin{proposition}\label{pr:stand}
Assume that equation \eqref{quadr} admits a real motherbody measure $\mu$. 
Then the following two conditions hold:  

\noindent 
{\rm (i)} any connected curve in the support of $\mu$   coincides with a horizontal trajectory of the quadratic differential
\begin{equation}\label{eq:maindiff}
\Theta=-\frac{D^2(z)}{P^2(z)}dz^2=\frac{4P(z)R(z)-Q^2(z)}{P^2(z)}dz^2. 
\end{equation}

\noindent
{\rm (ii)}  the support of $\mu$ should include all branching points of \eqref{quadr}.
\end{proposition}

\noindent
{\it Remark.} 
Observe that if $P(z)$ and $Q(z)$ are coprime, the set of all branching points coincides with the set of all zeros of $D(z)$.  In particular, in this case  requirement Proposition~\ref{pr:stand} implies that the set $DK_\Theta^0$ for the differential $\Theta$ should contain all  zeros  of $D(z)$.

\begin{proof} The fact that every curve in $\text{supp}(\mu)$ should coincide with some horizontal trajectory of \eqref{eq:maindiff} is well-known and follows from the  Plemelj-Sokhotsky's formula, see e.g. \cite{Pri}. It is based on the local observation that if a real measure $\mu=\frac{1}{\pi}\frac{\partial \C}{\partial \bar z}$ is supported
 on a smooth curve $\ga$, then the tangent to $\gamma$ at any  point $z_0\in \gamma$ should be perpendicular to $\overline {\C_1(z_0)}- \overline{\C_2(z_0)}$ where $\C_1$ and $\C_2$ are the one-sided limits of $\C$ when  $z\to z_0$, see e.g. \cite{BR}. (Here\;  $\bar {}$\;  stands for the usual complex  conjugation.) Solutions of \eqref{quadr} are given by $$\C_{1,2}=\frac{-Q(z)\pm \sqrt{Q^2(z)-4P(z)R(z)}}{2P(z)},$$  their difference being $$\C_1-\C_2=\frac{\sqrt{Q^2(z)-4P(z)R(z)}}{P(z)}.$$
The tangent line to the support of the real motherbody measure $\mu$ satisfying \eqref{quadr} at its arbitrary smooth point $z_0$,  being orthogonal to $\overline{ \C_1(z_0)}-\overline {\C_2(z_0)},$ is exactly given by the condition $\frac{4P(z_0)R(z_0)-Q^2(z_0)}{P^2(z_0)}dz^2>0$. The latter condition   defines the horizontal trajectory of $\Theta$ at $z_0$.


  Finally the observation that $\text{supp }\mu$ should contain all branching points of \eqref{quadr} follows immediately from the  fact that $\C_\mu$ is a well-defined function in $\bC\setminus  \text{supp }\mu$. 
\end{proof}
 
In many special cases  statements similar to  Proposition~\ref{pr:stand}  can be found in the literature, see e.g. recent \cite {AMFMGT} and references therein.  

\medskip
Proposition~\ref{pr:stand} allows us, under mild nondegeneracy assumptions, to formulate  necessary and sufficient conditions for the existence of a motherbody measure for \eqref{quadr}  which however are difficult to 
verify. Namely, let $\Ga\subset \bCP^1\times \bCP^1$ with coordinates $(\C,z)$ be the algebraic curve given by (the projectivization of) equation \eqref{quadr}. $\Ga$ has bidegree $(2,n+2)$ and is hyperelliptic. Let  $\pi_z:\Ga\to \bC$ be the projection of $\Ga$ on  the $z$-plane $\bCP^1$ along the $\C$-coordinate. From \eqref{quadr} we observe that $\pi_z$ induces a branched double covering of $\bCP^1$ by $\Ga$.  If $P(z)$ and $Q(z)$ are coprime and if $\deg D(z)=2n+2$,  the set of all branching points of $\pi_z: \Ga\to \bCP^1$  coincides with the set of all zeros of $D(z)$. (If $\deg D(z)<2n+2$ then $\infty$ is also a branching pont of $\pi_z$ of multiplicity $2n+2-\deg D(z)$.)   We need the following lemma. 

\begin{lemma}\label{lm:poles} If $P(z)$ and $Q(z)$ are coprime,  then at each pole of \eqref{quadr} i.e. at each zero  of $P(z)$, only one of two branches of $\Ga$ goes to  $\infty$. Additionally the residue of this branch at this zero equals that of $-\frac {Q(z)}{P(z)}$. 
\end{lemma}

\begin{proof} Indeed if $P(z)$ and $Q(z)$ are coprime, then no zero $z_0$ of $P(z)$ can be a branching point of \eqref{quadr} since $D(z_0)\neq 0$.  Therefore only one of two branches of $\Ga$ goes to $\infty$ at $z_0$. More exactly, the branch 
 $\C_1=\frac{-Q(z)+ \sqrt{Q^2(z)-4P(z)R(z)}}{2P(z)}$ attains a finite value at $z_0$ while the branch $\C_2=\frac{-Q(z)- \sqrt{Q^2(z)-4P(z)R(z)}}{2P(z)}$ goes to  $\infty$ where we use the agreement that $\lim_{z\to z_0} \sqrt{Q^2-4P(z)R(z)}=Q(z_0)$.  Now consider the residue of the branch $\C_2$ at $z_0$. Since residues depend continuously on the coefficients $(P(z),Q(z),R(z))$  it suffices to consider only the case when $z_0$ is a simple zero of $P(z)$.  Further if $z_0$ is a simple zero of $P(z)$ then 

$$Res(\C_2,z_0)= \frac {-2Q(z_0)} {2P^\prime(z_0)}= Res \left(-\frac{Q(z)}{P(z)}, z_0\right).$$
\end{proof}

\medskip 
By Proposition~\ref{pr:stand} (besides the obvious condition that \eqref{quadr} has a real branch near $\infty$ with the asymptotics $\frac{\al}{z}$ for some $\al\in \bR$) the necessary condition for \eqref{quadr} to admit a motherbody measure is that the set $DK_\Theta^0$ for the differential \eqref{eq:maindiff} contains all branching points of \eqref{quadr}, i.e. all zeros of $D(z)$.  Consider $\Ga_{cut}=\Ga\setminus \pi_z^{-1}(DK_\Theta^0)$. Since $DK_\Theta^0$ contains all branching points of $\pi_z$, $\Ga_{cut}$ consists of some number of open sheets each projecting diffeomorphically on its image in $\bCP^1\setminus DK^0_\Theta$. (The number of sheets in $\Ga_{cut}$ equals to twice the number of connected components in $\bC\setminus DK_\Theta^0$.)  Observe that since we have chosen a real branch of \eqref{quadr} at infinity with the asymptotics $\frac{\al}{z}$,  we have a marked point  $p_{br}\in \Ga$ over $\infty$. If we additionally assume that $\deg D(z)=2n+2,$ then $\infty$ is not a branching point of $\pi_z$  and   therefore $p_{br}\in \Ga_{cut}$. 

\begin{lemma}\label{lm:cut} If $\deg D(z)=2n+2$, then any choice of a spanning (multi)subgraph $G\subset DK_\Theta^0$ with no isolated vertices induces the unique choice of the section $S_G$ of $\Ga$ over $\bCP^1\setminus G$ which: 

\medskip
\noindent
a) contains $p_{br}$; b) is discontinuous at any point of $G$; c) is projected    by $\pi_z$ diffeomorphically onto $\bCP^1\setminus G$. 
\end{lemma}

Here by a spanning subgraph  we mean a subgraph containing all the vertices of the ambient graph. By a section of $\Ga$ over $\bCP^1\setminus G$ we mean a choice of one of two possible values of $\Ga$ at each point in $\bCP^1\setminus G$.

\begin{proof}
Obvious. 
\end{proof} 

Observe that the section $S_G$ might attain the value $\infty$ at some points, i.e. contain some poles of \eqref{quadr}. Denote the set of poles of $S_G$ by 
$Poles_G$. Now we can formulate our necessary and sufficient conditions. 

\begin{theorem}\label{th:necsuf} Assume that the following conditions are valid: 

\noindent
{\rm (i)} equation \eqref{quadr} has a real branch near $\infty$ with the asymptotic behavior  $\frac{\al}{z}$ for some $\al\in \bR$, comp. Lemma~\ref{lm:ProbBr};

\noindent
{\rm (ii)} $P(z)$ and $Q(z)$ are coprime, and the discriminant   $D(z)=Q^2(z)-4P(z)R(z)$ of equation \eqref{quadr} has degree $2n+2$;

\noindent
{\rm (iii)}  the set $DK_\Theta^0$ for  quadratic differential $\Theta$ given by \eqref{eq:maindiff} contains all  zeros of $D(z)$;

\noindent
{\rm (iv)} $\Theta$ has no closed horizontal trajectories.

Then \eqref{quadr} admits a real motherbody measure if and only if there exists a spanning  (multi)subgraph $G\subseteq DK^0_\Theta$ with no isolated vertices, such that 
all poles in $Poles_g$ are simple and all  their residues  are real, see notation above. 
\end{theorem}

\begin{proof} Indeed assume that $\eqref{quadr}$   satisfying {\rm(ii)} admits a real motherbody measure $\mu$. Assumption {\rm(i)} is obviously neccesary for the existence of a real motherbody measure and the necessity of  assumption {\rm(iii)} follows from Proposition~\ref{pr:stand} if {\rm(ii)} is satisfied.  The support of $\mu$ consists of a finite number of  curves and possibly a finite number of isolated points. Since each curve in the support of $\mu$ is a trajectory of $\Theta$ and $\Theta$ has no closed trajectories then the whole support of $\mu$ consists of double singular trajectories of $\Theta$ connecting its zeros, i.e. belongs to $DK_\Theta^0$. Moreover the support of $\mu$  should contain sufficently many double singular trajectories of $\Theta$ such that they include all the branching points of \eqref{quadr}. By {\rm(ii)} these are exactly all zeros of $D(z)$. Therefore the union of double singular trajectories of $\Theta$ belonging to the support of $\mu$ is a spanning (multi)graph of $DK^0_\Theta$ without isolated vertices. The isolated points in the support of $\mu$ are necessarily the poles of \eqref{quadr}. Observe that the Cauchy transform of any (complex-valued) measure can only have simple poles (as opposed to the Cauchy transform of a more general distribution). Since $\mu$ is real the residue of its Cauchy transform at each pole must be real as well.  Therefore existence of a real motherbody under the assumptions {\rm (i)}--{\rm(iv)} implies the existence of a spanning (multi)graph $G$ with the above properties. The converse is also immediate. 
\end{proof}

\noindent
{\it Remark.} Observe that if {\rm (i)} is valid then assumptions  {\rm (ii)} and {\rm (iv)} are generically satisfied. Notice however that {\rm (iv)} is violated in the special case when $Q(z)$ is absent  considered in subsection \ref{strb}. Additionally, if {\rm (iv)} is satisfied then the number of possible motherbody measures is finite.  On the other hand, it is the assumption {\rm (iii)} which  imposes severe additional restrictions on admissible triples $(P(z),Q(z),R(z))$. At the moment the authors have  no information about possible cardinalities of the sets $Poles_G$ introduced above. Thus it is difficult to estimate the number of conditions required for \eqref{quadr} to admit a motherbody measure. 
 Theorem~\ref{th:necsuf} however leads to the following sufficient condition for the existence of a real motherbody measure for \eqref{quadr}.

\begin{corollary}\label{cor:suf} If, additionally to assumptions {\rm (i)}--{\rm (iii)} of Theorem~\ref{th:necsuf}, one assumes  that all roots of $P(z)$ are simple and all residues of $\frac{Q(z)}{P(z)}$ are real  then \eqref{quadr} admits a real motherbody measure.  
\end{corollary}

\begin{proof} Indeed if all roots of $P(z)$ are simple and all residues of $\frac{Q(z)}{P(z)}$ are real then all  poles of \eqref{quadr} are simple with real residues.  In this case for any choice of a  spanning (multi)subgraph $G$ of $DK_\Theta^0$, there exists a real motherbody measure whose support coincides with $G$ plus possibly some poles of \eqref{quadr}.  Observe that if all roots of $P(z)$ are simple and all residues of $\frac{Q(z)}{P(z)}$ one can omit assumption {\rm (iv)}. In case when $\Theta$ has no closed trajectories then all possible real motherbody measures are in a bijective correspondence with all spanning (multi)subgraphs of $DK_\Theta^0$ without isolated vertices. In the opposite case such measures are in a bijective correspondence with the unions of  a spanning (multi)subgraph of $DK_\Theta^0$ and an arbitrary (possibly empty)  finite collection of closed trajectories. 
\end{proof} 

Although we at present do not have rigorous results about the structure of the set of general equations \eqref{quadr} admiting a real motherbody measure, we think  that generalizing our previous methods and statements from \cite{BBS}, \cite{HS1} and \cite{STT}, one would be able to settle the following conjecture. 

\begin{conjecture}\label{conj:positive} Fix any monic polynomial $P(z)$ of degree $n+2$ and an arbitrary polynomial $Q(z)$ of degree at most $n+1$. Let $\Omega_{P,Q}$ denote the set of all polynomials $R(z)$ of degree at most $n$  such that \eqref{quadr} admits a probability measure, i.e  a positive motherbody measure of mass $1$. Then $\Omega_{P,Q}$ is a real semi-analytic variety of (real) dimension $n$.
\end{conjecture}

\medskip
\subsection {Case $Q=0$ and Strebel differentials}\label{strb}

In this subsection we present in more detail the situation when the middle term in \eqref{quadr} is vanishing, i.e. 
$Q(z)=0$.  In this case one can obtain complete information about the number of possible motherbody signed measures and also a criterion of the existence of a positive motherbody measure. This material is mainly borrowed from \cite{STT}.

\medskip
The following statement  can be easily derived from results of Ch. 3, \cite{Str}.  

\begin{lemma}\label{lm:quadrpoles}
\label{lemma1} If a meromorphic quadratic   differential $\Psi$
 is Strebel, then it has no poles of order
greater than 2. If it has a pole of order 2, then the coefficient  of the 
leading term  of $\Psi$ at this pole is negative.
\end{lemma}

In view of this lemma let us introduce the class $\M$ of meromorphic  quadratic  differentials on a Riemann surface $Y$ satisfying  the above restrictions, i.e. their poles are at most of order $2$ and at each such double pole the leading coefficient is negative.

It is known that  for a meromorphic Strebel differential $\Psi$ given on a compact Riemann surface $Y$  
without boundary the set $K_\Psi$ has several nice properties. In particular, it is a finite embedded multigraph on $Y$ whose edges are double singular trajectories. In other words, for a Strebel quadratic differential $\Psi$, one gets $K_\Psi=DK_\Psi$.  
   

 Our next result relates a Strebel differential $\Psi$ on $\bCP^1$ with a double pole at $\infty$ to real-valued measures supported on $K_\Psi$.

 \begin{theorem}\label{th:charac} \rm{(i)} Given two coprime  polynomials $P(z)$ and $R(z)$ of degrees $n+2$ and $n$ respectively where $P(z)$ is monic and $R(z)$ has a negative leading coefficient, the algebraic function given by the equation 
 \begin{equation}\label{eq:quadr}
 P(z)\C^2+R(z)=0
 \end{equation}
admits a motherbody measure $\mu_\C$ if and only if the quadratic differential $\Psi=R(z)dz^2/P(z)$ is Strebel. 

\medskip
\noindent
\rm{(ii)} Such motherbody measures are, in general, non-unique. If we additionally require that  the support of each such measure consists only of double singular trajectories, i.e. is a spanning subgraph of $K_\Psi=DK_\Psi,$ then  for any $\Psi$ as  above there exists exactly $2^{d-1}$  real measures 
where $d$ is the total number of connected components in $\bCP^1\setminus K_\Psi$ (including the unbounded component, i.e. the one containing $\infty$). These measures are in $1-1$-correspondence with $2^{d-1}$ possible choices of the branches of $\sqrt{-{R(z)}/{P(z)}}$  in the union of $(d-1)$ bounded components of $\bCP^1\setminus K_\Psi$. 
 \end{theorem}
 
To prove Theorem~\ref{th:charac} we need  some information about compactly supported real
measures and their Cauchy transforms.
 
 \begin{center}
\begin{picture}(440,270)(0,0)
\put(170,240){\line(1,0){60}}
\put(170,240){\circle*{3}}
\put(230,240){\circle*{3}}

\qbezier(200,275)(250,270)(230,240)
\qbezier(200,275)(150,270)(170,240)
\qbezier(200,205)(250,210)(230,240)
\qbezier(200,205)(150,210)(170,240)

\put(180,255){\circle*{3}}
\put(220,255){\circle*{3}}
\qbezier(180,255)(200,265)(220,255)

\put(180,225){\circle*{3}}
\put(220,225){\circle*{3}}
\qbezier(180,225)(200,215)(220,225)

\put(20,110){\line(1,0){60}}
\put(20,110){\circle*{3}}
\put(80,110){\circle*{3}}
\put(55,110){$\ominus$}

\put (50,150) {\vector(0,-1){6}}
\put (50,137) {\vector(0,-1){6}}
\put (50,123) {\vector(0,1){6}}

\put (50,97) {\vector(0,-1){6}}
\put (50,83) {\vector(0,1){6}}
\put (50,70) {\vector(0,1){6}}

\put(30,125){\circle*{3}}
\put(70,125){\circle*{3}}
\qbezier(30,125)(50,135)(70,125)
\put(55,128){$\oplus$}

\put(30,95){\circle*{3}}
\put(70,95){\circle*{3}}
\qbezier(30,95)(50,85)(70,95)
\put(55,90){$\oplus$}

\put(15,30){1st measure}

\put(120,110){\circle*{3}}
\put(180,110){\circle*{3}}

\qbezier(150,145)(200,140)(180,110)
\put(155,128){$\ominus$}
\qbezier(150,145)(100,140)(120,110)
\put(155,90){$\oplus$}

\put (150,152) {\vector(0,-1){6}}
\put (150,138) {\vector(0,1){6}}
\put (150,120) {\vector(0,-1){6}}

\put (150,97) {\vector(0,-1){6}}
\put (150,83) {\vector(0,1){6}}
\put (150,70) {\vector(0,1){6}}

\put(130,125){\circle*{3}}
\put(170,125){\circle*{3}}
\qbezier(130,125)(150,135)(170,125)

\put(130,95){\circle*{3}}
\put(170,95){\circle*{3}}
\qbezier(130,95)(150,85)(170,95)
\put(155,140){$\oplus$}

\put(115,30){2nd measure}

\put(220,110){\circle*{3}}
\put(280,110){\circle*{3}}

\qbezier(250,75)(300,80)(280,110)
\put(255,128){$\oplus$}
\qbezier(250,75)(200,80)(220,110)
\put(255,90){$\ominus$}

\put (250,150) {\vector(0,-1){6}}
\put (250,137) {\vector(0,-1){6}}
\put (250,123) {\vector(0,1){6}}

\put (250,98) {\vector(0,1){6}}
\put (250,84) {\vector(0,-1){6}}
\put (250,68) {\vector(0,1){6}}

\put(230,125){\circle*{3}}
\put(270,125){\circle*{3}}
\qbezier(230,125)(250,135)(270,125)

\put(230,95){\circle*{3}}
\put(270,95){\circle*{3}}
\qbezier(230,95)(250,85)(270,95)
\put(255,75){$\oplus$}

\put(215,30){3rd measure}

\put(320,110){\line(1,0){60}}
\put(320,110){\circle*{3}}
\put(380,110){\circle*{3}}

\qbezier(350,145)(400,140)(380,110)
\put(355,128){$\ominus$}
\qbezier(350,145)(300,140)(320,110)
\put(355,90){$\ominus$}
\qbezier(350,75)(400,80)(380,110)
\put(355,143){$\oplus$}
\qbezier(350,75)(300,80)(320,110)
\put(355,110){$\oplus$}

\put (350,152) {\vector(0,-1){6}}
\put (350,136) {\vector(0,1){6}}
\put (350,122) {\vector(0,-1){6}}

\put (350,98) {\vector(0,1){6}}
\put (350,84) {\vector(0,-1){6}}
\put (350,68) {\vector(0,1){6}}

\put(330,125){\circle*{3}}
\put(370,125){\circle*{3}}
\qbezier(330,125)(350,135)(370,125)
\put(355,75){$\oplus$}

\put(330,95){\circle*{3}}
\put(370,95){\circle*{3}}
\qbezier(330,95)(350,85)(370,95)

\put(315,30){4th measure}

\put(00,00){Figure 2. The set $K_\Psi=DK_\Psi$ of a Strebel differential $\Psi$  and 4 related real measures.}

\label{figNew}
\end{picture}
\end{center}

\begin{lemma} [comp. Th. 1.2, Ch. II, \cite{Ga}]
\label{lemma4} Suppose $f\in{}L_{loc}^1(\mathbb{C})$ and that
$f(z)\rightarrow0$ as $z\rightarrow\infty$ and let $\mu$ be a
compactly supported measure in $\mathbb{C}$ such that
$$\mu=\frac{1}{\pi}\frac{\partial f}{\partial\bar{z}}$$ in the sense of
distributions. Then $f(z)=C_{\mu}(z)$ almost everywhere, where
$C_{\mu}(z)=\int_{\mathbb{C}}{d\mu(\xi)}/{(z-\xi)}$ is the Cauchy
transform of $\mu$.
\end{lemma}

\begin{proof}
It is clear that $C_{\mu}$ is locally integrable, analytic off the
closure of the  support of $\mu$ and vanishes at infinity. Consider 
$h=f-C_{\mu}$ and notice  that $h$ is a locally integrable function
vanishing at infinity and satisfying
${\partial{}h}/{\partial\bar{z}}=0$ in the sense of
distributions. We must show that $h\equiv 0$ almost everywhere. Let
$\phi_r\in{}C_0^{\infty}(\mathbb{C})$ be an approximate to the
identity, i.e. $\phi_r\geq0$, $\int_{\mathbb{C}}\phi_rdxdy=1$ and
$supp({\phi_r})\subset\{|z|<r\}$, and consider the convolution
\[h_r(z)=h*\phi_r=\int_{\mathbb{C}}h(z-w)\phi_r(w)dxdy, \ w=x+iy.\]
It is well-known that $h_r\in{}C^{\infty}$ and that
$\lim\limits_{r\rightarrow0}h_r\rightarrow{}h$ in $L^1(K)$ for any
compact set $K$. Moreover
\[\frac{\partial{}h_r}{\partial\bar{z}}=\frac{\partial{}h}{\partial\bar{z}}*\phi_r=0.\]
This shows that $h_r$ is an entire function which vanishes at
infinity, implying that $h_r\equiv0$. Hence $h\equiv 0$ a.e.
\end{proof}

Lemma  \ref{lemma4} shows how given a Strebel differential  $\Psi= R(z)dz^2/{P(z)}$ to construct real  measures whose support lies in   $K_\Psi$  and whose Cauchy transforms satisfy \eqref{eq:quadr}  by
specifying   branches of $\sqrt{-{R(z)}/{P(z)}}$ in
$\mathbb{C}\setminus{}K_\Psi$. 
 
 \begin{proof} [Proof of Theorem~\ref{th:charac}] 
 To settle part (i)  we first show  that  for a given  Strebel differential 
$\Psi=R(z)dz^2/P(z)$ one can construct  real measures supported on  $K_\Psi$ with the required properties. To do this choose  arbitrarily a branch of $\sqrt{-{R(z)}/{P(z)}}$ in each bounded connected component of  $\mathbb{C}\setminus K_\Psi$. (In the unbounded connected component we have to choose the branch which behaves as $\alpha/{z}$ at infinity with $\alpha>0$; such a choice is possible by our assumptions on the leading coeffcients of $P(z)$ and $R(z)$.) Define now 
$$\mu:=\frac{\partial\sqrt{-{R(z)}/{P(z)}}}{\partial\bar{z}}$$
 in the sense of distributions. The distribution $\mu$ is evidently compactly supported on $\mathfrak F$.
By lemma \ref{lemma4} we get that $C_{\mu}(z)$ satisfies \eqref{eq:quadr} a.e in $\bC$. 

It remains to show
that $\mu$ is a real measure. Take a point $z_0$ in the support of
$\mu$ which is a regular point of ${R(z)}dz^2/{P(z)}$ and take a small
neighborhood $N_{z_0}$ of $z_0$ which does not contain  roots of
$R(z)$ and $P(z)$. In this neighborhood we can choose a single branch
$B(z)$ of $\sqrt{-{R(z)}/{P(z)}}$. Notice that $N_{z_0}$ is divided
into two sets by the support of $\mu$ since it by
construction consists of singular trajectories of $\mu$. Denote  these
 sets  by $M$ and $M'$ resp. Choosing $M$ and $M'$
appropriately we can represent $C_{\mu}$ as $\chi_{M}B-\chi_{M'}B$
in $N_{z_0}$ up to the support of $\mu$, where $\chi_{X}$ denotes
the characteristic function of the set $X$. By  theorem 2.15 in
\cite{Her}, ch.2  we have
\[<\mu,\phi>=<\frac{\partial{}C_{\mu}}{\partial\bar{z}},\phi>=<\frac{\partial(\chi_{M}B-\chi_{M'}B)}{\partial\bar{z}}>=i\int_{\partial{}M}B(z)\phi{}dz,
\]
for any test function $\phi$ with compact support in $N_{z_0}$.
Notice that the last equality holds because $\phi$ is identically
zero in a neighborhood of $\partial{}N_{z_0}$, so it is only on the
common boundary of $Y$ and $Y'$ that we get a contribution to the
integral given in the last equality. This common boundary  is the singular trajectory $\gamma_{z_0}$ intersected with the neighborhood $N_{z_0}$. The integral
\[i\int_{\partial{}M}B(z)\phi{}dz\]
is real since the change of coordinate $w=\int_{z_0}^{z}iB(\xi)d\xi$
transforms the integral to the integral of $\phi$ over a part of the
real line. This shows that $\mu$ is locally a real measure, which 
proves one implication in part (i) of the theorem.

To prove  the converse implication, i.e. that a compactly supported real 
measure $\mu$  whose Cauchy transform satisfies \eqref{eq:quadr} everywhere except for a set of measure zero produces the Strebel differential $\Psi={R(z)}dz^2/{P(z)}$ consider 
 its logarithmic potential $u_{\mu}(z)$, i.e. 
\[u_{\mu}(z):=\int_{\mathbb{C}}\log|z-\xi|d\mu(\xi).\]
The function $u_{\mu}(z)$ is harmonic outside the
support of $\mu$, and  subharmonic in the whole $\mathbb{C}$. The following  relation mentioned in the introduction  
\[C_{\mu}(z)=\frac{\partial{u_{\mu}(z)}}{\partial{z}}\]
connects  $u_{\mu}$ and $C_{\mu}$. 
It implies that the set of level curves of $u_{\mu}(z)$ coincides 
with the set of horizontal trajectories of the quadratic differential
$-C_{\mu}(z)^2dz^2$. Indeed, the gradient of $u_{\mu}(z)$ is given
by the vector  field with coordinates $\left({\partial{u_{\mu}}}/{\partial{x}},{\partial{u_{\mu}}}/{\partial{y}}\right)$ in $\bC$. Such a vector in $\bC$ coincides with the complex number
$2{\partial{u_{\mu}}}/{\partial\bar{z}}$. Hence, the gradient of
$u_{\mu}$ at $z$ equals to $\bar{C}_{\mu}(z)$ (i.e. the complex conjugate of
$C_{\mu}(z)$). But this is the same as saying that the vector
$i{}\bar{C}_{\mu}(z)$ is orthogonal to (the tangent line to) the level
curve of $u_{\mu}(z)$ at every point $z$ outside the support of
$\mu$. Finally, notice that at each point $z$ one has 
$$-C_{\mu}^2(z)(i\bar{C}^2_{\mu}(z))>0,$$
 which  means
 that the horizontal trajectories of $-C_{\mu}(z)^2dz^2={R(z)}dz^2/{P(z)}$ are
 the level curves of $u_{\mu}(z)$ outside the support of
$\mu$.  Notice that $u_{\mu}(z)$ behaves as $\log {|z|}$ near $\infty$ and  is continuous except for
 possible second order poles where it has logarithmic singularities with a negative leading coefficient. This guarantees that almost all its level curves are closed and smooth  implying that  ${R(z)}dz^2/{P(z)}$ is Strebel. 

Let us  settle part (ii) of  Theorem~\ref{th:charac}.  Notice that if a real measure whose Cauchy transform satisfies  \eqref{eq:quadr} a.e. is supported on the compact set $K_\Psi$ (which consists of finitely many singular trajectories and singular points)  all one needs to determine it uniquely is just  to prescribe which of the two branches of $\sqrt{-{R(z)}/{P(z)}}$ one should choose as the Cauchy transform of this measure  in each bounded connected component of the complement $\bC\setminus K_\Psi$. (The choice of a branch in the unbounded component is already prescribed by the requirement that it should behave as $\frac{\alpha}{z}$ near $\infty$ where $\alpha>0$.)  Any such choice of branches in open domains leads to a real measure, see the above proof of part (i).  \end{proof}

 Concerning possible  positive measures   we can formulate an exact criterion of the existence of a positive measure for a rational quadratic differential $\Psi=R(z)dz^2/P(z)$ in terms of rather simple topological properties of $K_\Psi$.   To do this we need a few definitions. Notice that $K_\Psi$ is a planar multigraph with the following properties.  The vertices of $K_\Psi$ are the finite singular points of $\Psi$ (i.e. excluding $\infty$)  and its edges are singular trajectories connecting these finite singular points. Each (open) connected component of  $\bC\setminus K_\Psi$ is homeomorphic to an (open) annulus. $K_\Psi$ might have isolated vertices which are the finite double poles of $\Psi$. Vertices of $K_\Psi$ having valency $1$ (i.e. hanging vertices) are exactly the simple poles of $\Psi$. Vertices different from the isolated and hanging vertices are the zeros of $\Psi$. The number of  edges adjacent to a given vertex minus $2$ equals the order of the  zero of $\Psi$ at this point.  Finally,  the sum of the  multiplicities of all poles (including the one at $\infty$) minus the sum of the multiplicities of all zeros equals $4$.

By a {\em simple cycle} in a planar multigraph $K_\Psi$ we mean any closed non-selfintersecting curve formed by  the edges of $K_\Psi$. (Obviously, any simple cycle bounds an open domain homeomorphic to a disk which we call the {\em interior of the cycle}.) 

\begin{proposition}\label{pr:crit} A Strebel differential $\Psi={R(z)}dz^2/{P(z)}$ admits a positive motherbody measure  if and only if no edge of $K_\Psi$ is attached to a simple cycle from inside. In other words,  for any simple cycle in $K_\Psi$ and any edge not in the cycle but adjacent to some vertex in the cycle,  
this edge does not belong to its interior. The support of the positive measure coincides with the forest obtained from $K_\Psi$ after the removal of all its simple cycles.  
 \end{proposition}

\noindent
{\it Remark.} Notice that under the assumptions of Proposition~\ref{pr:crit} all simple cycles of $K_\Psi$ are pairwise non-intersecting and, therefore, their removal is well-defined in an unambiguous way.

In particular, the compact on the right part of Fig.~3 admits no positive measure since it contains an edge cutting a simple cycle (the outer boundary) in two smaller cycles.  The left picture on Fig.~4 has no such edges and, therefore, admits a positive measure whose support consists of the four horizontal edges of $K_\Psi$. 

\begin{center}
\begin{picture}(440,160)(0,0)

\put (85,120){\circle{25}}

\put(80,120){\circle*{3}}
\put(80,120){\line(1,0){10}}
\put(90,120){\circle*{3}}

\put (125,120){\circle{25}}
\put(120,120){\circle*{3}}
\put(120,120){\line(1,0){10}}
\put(130,120){\circle*{3}}

\put(97,120){\circle*{3}}
\put(97,120){\line(1,0){16}}
\put(113,120){\circle*{3}}

\qbezier(60,120)(75,195)(150,120)
\qbezier(60,120)(75,45)(150,120)
\put(150,120){\circle*{3}}
\put(150,120){\line(1,0){20}}
\put(170,120){\circle*{3}}

\put(230,120){\line(1,0){60}}
\put(230,120){\circle*{3}}
\put(290,120){\circle*{3}}

\qbezier(260,145)(310,175)(290,120)
\qbezier(260,145)(210,175)(230,120)
\qbezier(260,85)(310,90)(290,120)
\qbezier(260,85)(210,90)(230,120)
\put(260,145){\circle*{3}}
\put(260,145){\line(0,-1){12}}
\put(260,133){\circle*{3}}

\put(280,130){\circle*{3}}
\put(240,130){\circle*{3}}
\qbezier(240,130)(260,122)(280,130)

\put(240,105){\circle*{3}}
\put(280,105){\circle*{3}}
\qbezier(240,105)(260,95)(280,105)





\put(00,40){Figure 3. Two examples of $K_\Psi$ admitting and not admitting a positive measure.}
\end{picture}
\end{center}


\begin{proof} [Proof of Proposition~\ref{pr:crit}]

Notice that given a  finite measure supported on a finite union of curves with continuous density we get that its logarithmic potential will be a continuous function. As we have shown in the proof of part (i) of 
Theorem~\ref{th:charac},  in our situation the one-sided limits of the gradient of the logarithmic potential are  orthogonal to the tangent lines of the curves in the support. In other words, the logarithmic potentials of our real measures attain constant values on each connected component of the support. (This phenomenon is characteristic for the so-called equilibrium measures of a given collection of curves, see \cite {SaTo}.) Moreover, if  the considered  measure $\mu$ is positive (resp. negative) on a given curve in its support then its potential attains a local maximum (resp. minimum) on this curve. Thus for a positive measure as above its potential has no local minima except at $\infty$. As a direct corollary of the latter observation we get that the support of a positive measure as above can not contain cycles. Indeed, let it contain a cycle. Without  loss of generality we can assume that this cycle is simple (i.e. does not have self-intersections) since every cycle consists of simple cycles. Consider the interior of this simple cycle. On its boundary the potential is constant and is locally decreasing in the direction pointing inside this cycle. Therefore, the potential must have a local minimum in the interior which is impossible. Also notice that if a required positive measure exists then its potential should increase on each simple cycle in the direction of its interior. 

Let us now show that the existence of a positive measure implies that no edge of $K_\Psi$ is attached to a simple cycle from inside. Indeed, assume that such an edge exists. The potential should be constant on each  connected component of $K_\Psi$ and, in particular, on the one containing the considered cycle and the extra edge attached  to it. Finally,  it should increase in the direction of the interior of the cycle. But this immediately implies that the potential attains a local minimum on this edge which means that the (density of the) measure on this extra edge is negative, see Fig.~2 where the arrows show the directions of the gradient.

Let us show the converse implication, i.e. that the absence of such edges implies the existence of a positive measure. Notice that the assumptions of  Proposition~\ref{pr:crit} are equivalent to the fact that any connected component of the graph $K_\Psi$ has the following property. No edge belonging to a component  is located inside the interiors of the cycles belonging to this component (if they exist). Thus we can uniquely define the branch of $\sqrt{-{R(z)}/{P(z)}}$ in each connected component of $\bC\setminus K_\Psi$ so that on the boundary of each simple cycle in $K_\Psi$ the gradient of the logarithmic potential  points inside the cycle. (Recall that the gradient coincides with the complex conjugate of the chosen branch of $\sqrt{-{R(z)}/{P(z)}}$ and  is orthogonal to the edges of $K_\Psi$ at each point except the vertices.)  Since no simple cycles belonging to  the same connected component of $K_\Psi$ lie within each other this choice is unique and well-defined. It leads to the positive measure supported on the complement to the union of all simple cycles of $K_\Psi$. 
\end{proof}



Let us finish the paper with the following important observation. 

\begin{proposition}\label{prop:density}
For any monic $P(z)$ of degree $n+2$ the set of all polynomials $R(z)$ of degree $n$ and with leading coefficient $-1$ such that the differential $\frac{R(z)dz^2}{P(z)}$ is Strebel is dense in the space of all polynomials of degree $n$  with leading coefficient $-1$. In fact,  this set  is the countable union of real semi-analytic varieties of positive codimension. 
\end{proposition}

This circumstance illustrates  the difficulty of the general problem to determine for which algebraic equations a real motherbody measure exists. 

\section {Final remarks}

\noindent 
{\bf 1.} The natural question about which algebraic functions of degree bigger than $2$ whose Newton polygon intersects the diagonal in the $(C,z)$-plane nontrivially admit a real motherbody measure is hard to answer.  Some steps in this direction can be found in \cite{HS1}. This topic is apparently closely related to the (non-existing) notion of  Strebel differential of higher order which we hope to develop in the future. In any case, it is clear that no results similar to Theorem~\ref{th:main} are possible and one needs to impose highly non-trivial additional  conditions on such functions  to ensure the existence of a probability measure.

\medskip
\noindent 
{\bf 2.}  {\rm Problem.} Given a finite set $S$ of monomials satisfying the assumptions of Lemma~\ref{lm:ProbBr} and Lemma~\ref{lm:irred}, consider the linear space $Pol_S$ of all polynomials $P(\C,z)$ whose Newton polygon is contained in $S$.  What is the (Hausdorff) dimension of the subset $MPol_S\subseteq Pol_S$ of  polynomials admitting a motherbody measure?  
\medskip 





\begin{thebibliography}{30}

\bibitem{AMFMGT} M. J. Atia, A. Mart\'inez-Finkelshtein, P. Mart\'inez-Gonz\'alez, and F. Thabet, Quadratic differentials and asymptotics of Laguerre polynomials with varying complex parameters, arXiv: 1311.0372.  

\bibitem{Bar} Y.~Baryshnikov, On Stokes Sets,  New Developments in Singularity Theory
NATO Science Series Volume 21, 2001, pp 65--86. 





\bibitem{Be} S.~R.~Bell,  The Cauchy transform, potential theory, and conformal mapping. Studies in Advanced Mathematics. CRC Press, Boca Raton, FL, 1992. x+149 pp.

\bibitem{BOS}  O. B. Bekken, B. K. Oksendal and A. Stray (Editors), Spaces of analytic functions, Lecture
Notes in Math., vol. 512, Springer-Verlag, Berlin, 1976.

\bibitem{Ber} T.~Bergkvist, {\em  On asymptotics of polynomial eigenfunctions for exactly solvable differential operators}. J. Approx. Theory 149 (2007), no. 2, 151--187. 

\bibitem{BR} T.~Bergkvist and H.~Rullg\aa rd, {\em  On polynomial
eigenfunctions for a class of differential operators.} Math. Res. Lett. 9 (2002), 153--171.

\bibitem{BB} J.~Borcea, R.B\o gvad, {\em Piecewise harmonic subharmonic functions and positive Cauchy transforms.} Pacific J. Math., Vol. 240 (2009), No. 2, 231--265.

\bibitem{BBB} J.-E.~Bj\"ork, J.~Borcea, R.B\o gvad, {\em Subharmonic Configurations and
Algebraic Cauchy Transforms of Probability Measures.} Notions of Positivity and the Geometry of Polynomials
Trends in Mathematics 2011, pp 39--62. 

\bibitem {BBS} J.~ Borcea, R.~B\o gvad and  B.~Shapiro,  Homogenized spectral pencils for exactly solvable operators: asymptotics of polynomial eigenfunctions,  Publ. RIMS, vol 45 (2009) 525--568.   Corrigendum: ``Homogenized spectral pencils for exactly solvable operators: asymptotics of polynomial eigenfunctions",  Publ. RIMS, vol 85 (2012) 229--233.  

\bibitem{BroSt} M.~A.~Brodsky and V.~N.~Strakhov, {\em On the Uniqueness of the Inverse Logarithmic Potential Problem.}  
 SIAM Journal on Applied Mathematics, Vol. 46, No. 2 (Apr., 1986), pp. 324--344.

\bibitem{CMR} J.~A.~Cima, A.~L.~ Matheson, W.~T.~Ross, The Cauchy transform. Mathematical Surveys and Monographs, 125. American Mathematical Society, Providence, RI, 2006. x+272 pp.


\bibitem{GaSj1} S.~J.~Gardiner and T. Sj\"odin, {\em Convexity and the exterior inverse problem of potential theory.}  Proceedings of the AMS, 136:5 (2008), 1699--1703.

\bibitem{GaSj2} S.~J.~Gardiner and T. Sj\"odin, {\em Partial balayage and the  exterior inverse problem of  potential theory.} Potential Theory and Stochastics in Albac, 111-123, Bucharest, Theta (2009). 

\bibitem{Ga} J.B.~Garnett, {Analytic capacity and measure}, LNM 297,  Springer-Verlag, 1972, 138 pp.


\bibitem{Gu} B.~Gustafsson,  {\em On Mother Bodies of Convex Polyhedra.} SIAM J Math. Anal., 29:5, (1998) 
1106--1117.





\bibitem{Her} L.~H\"ormander, The analysis of linear partial differential operators. I. Distribution theory and Fourier analysis. Reprint of the second (1990) edition. Classics in Mathematics. Springer-Verlag, Berlin, 2003. x+440 pp.

\bibitem  {HS1} T.~Holst, B.~Shapiro, {\em On higher Heine-Stieltjes polynomials.} Isr. J. Math. 183 (2011) 321--347. 

\bibitem{KMWWZ} I.~Krichever, M.~Mineev-Weinstein, P.~Wiegmann and A.~Zabrodin, {\em Laplacian growth
and Whitham equations of soliton theory.} Phys. D, 198(2004), no. 1-2, 1--28.

 \bibitem{MFR1}  A.~Mart\'inez-Finkelshtein, E.~A.~Rakhmanov, 
 {\em Critical measures, quadratic differentials, and weak limits of zeros of Stieltjes polynomials}, Commun. Math. Phys.  vol. 302 (2011) 53--111.
 
 \bibitem{MFR2}  A.~Mart\'inez-Finkelshtein, E.~A.~Rakhmanov, 
 {\em On asymptotic behavior of Heine-Stieltjes and Van Vleck polynomials}, Contemp. Math. vol. 507, (2010) 209--232.

\bibitem{Mu} T.~Murai,  A real variable method for the Cauchy transform, and analytic capacity. Lecture Notes in Mathematics, 1307. Springer-Verlag, Berlin, 1988. viii+133 pp.

\bibitem{No} P. S. Novikoff, {\em Sur le probl\'eme inverse du potentiel}, C. R. (Dokl.) Acad. Sci. URSS, (N.S.) 18 (1938), 165-168.

\bibitem{PaSh} D.~Pasechnik, B.~Shapiro,  On polygonal measures with vanishing harmonic moments, Journal d'Analyse Math., to appear. 

\bibitem{Pri} I.~Pritsker, How to  find a measure from its potential, CMFT, vol. 8(2), 2008, 597--614. 

\bibitem{Rans} T.~Ransford, Potential theory in the complex plane, LMS Student Texts 28, 1995, ix+234 pp. 

\bibitem{SaTo} E.~B.~Saff, V.~Totik, Logarithmic potentials with external fields.
Appendix B by Thomas Bloom. Grundlehren der Mathematischen Wissenschaften [Fundamental Principles of Mathematical Sciences], 316. Springer-Verlag, Berlin, (1997), xvi+505 pp.

\bibitem{Sa} M.~Sakai,  {\em A moment problem on Jordan domains.}  Proc. Amer. Math. Soc. 70 (1978), no. 1, 35--38. 

\bibitem{SaStSha} T.~V.~Savina, B.~Yu.~ Sternin, V.~E.~ Shatalov,  {\em On a minimal element for a family of bodies producing the same external gravitational field.}  Appl. Anal. 84 (2005), no. 7, 649--668.

\bibitem {Sh1} B.~Shapiro, {\em Algebro-geometric aspects of Heine-Stieltjes theory.}   J. London Math. Soc.  83(1) (2011) 36--56.  

\bibitem{STT} B.~Shapiro, K.~Takemura, M.~Tater, {\em On spectral polynomials of the Heun equation. II.}  II,   Comm. Math. Phys. 311(2) (2012),  277--300.

\bibitem{HSh} H.~S.~Shapiro,  The Schwarz Function and its Generalization to Higher Dimensions. 1992, 
(Univ. of Arkansas Lect. Notes Math. Vol. 9, Wiley N.Y.)

\bibitem{Sj} T.~Sj\"odin, {\em Mother Bodies of Algebraic Domians in the Complex Plane.} 
Complex Variables and Elliptic Equations, 51:4  (2006) 357--369. 

\bibitem{Str} K.~Strebel, {Quadratic differentials}, Ergebnisse der Mathematik und ihrer Grenzgebiete, 5, Springer-Verlag, Berlin,  (1984), xii+184 pp.

\bibitem{Za} L.~Zalcmann, Analytic capacity and rational approximation. Lecture Notes in Mathematics, No. 50 Springer-Verlag, Berlin-New York 1968 vi+155 pp. 

\bibitem {Zi} D.~Zidarov,  Inverse Gravimetric Problem in Geoprospecting and Geodesy  1990, (Amsterdam: Elsevier.) vii+283 pp. 



\end{thebibliography}
\end{document}